\documentclass[graybox]{svmult}

\usepackage{amsmath}
\usepackage{eucal}
\usepackage[active]{srcltx}
\usepackage{amssymb}
 \usepackage[usenames,dvipsnames]{pstricks}
 \usepackage{epsfig}
 \usepackage{mathrsfs}
\usepackage{mathpazo}
\usepackage{tikz}
\usetikzlibrary{backgrounds}
\usetikzlibrary{patterns,fadings}
\usetikzlibrary{arrows,decorations.pathmorphing}
\usetikzlibrary{calc}
\definecolor{light-gray}{gray}{0.95}
\usepackage{graphicx}
\usepackage{epsfig}
\usepackage[ansinew]{inputenc} %para acentos em portugues
\usepackage{mathptmx}       % selects Times Roman as basic font
\usepackage{helvet}         % selects Helvetica as sans-serif font
\usepackage{courier}        % selects Courier as typewriter font
   \usepackage{wasysym}
\usepackage{makeidx}         % allows index generation
\usepackage{graphicx}        % standard LaTeX graphics tool
\usepackage{multicol}        % used for the two-column index

\usepackage{url}
\usepackage{bbm}

\newcommand{\mc}[1]{{\mathcal #1}}

\newcommand{\bb}[1]{{\mathbb #1}}

\makeindex             % used for the subject index
\date{\today}
\begin{document}

\title*{Derivation of the Stochastic Burgers equation from the WASEP}
\author{Patr\' icia Gon\c calves}
\institute{Patrícia Gonçalves \at Departamento de Matem\'atica, PUC-RIO, Rua Marqu\^es de S\~ao Vicente, no. 225, 22453-900, Rio de Janeiro, Rj-Brazil and
CMAT, Centro de Matem\'atica da Universidade do Minho, Campus de Gualtar, 4710-057 Braga, Portugal.
\email{patricia@mat.puc-rio.br } }

\maketitle

\abstract*{In these notes we give a simple proof of the second-order Boltzmann-Gibbs Principle, which is the main tool in order to prove that the equilibrium fluctuations of the WASEP are given, in the regime of the critical strength asymmetry, by the Stochastic Burgers equation.}

\abstract{In these notes we give a simple proof of the second-order Boltzmann-Gibbs Principle, which is the main tool in order to prove that the equilibrium fluctuations of the WASEP are given, in the regime of the critical strength asymmetry, by the Stochastic Burgers equation.}

\section{Introduction}
The KPZ equation was proposed in \cite{KPZ} as the default stochastic partial differential equation ruling the evolution of the profile of a growing interface. For a time $t$ and a space variable $x$, if $h_t(x)$ denotes the height of that interface at $t$ and $x$, then the KPZ equation reads as
\begin{equation*}
\label{KPZ0}
d h_t = A\Delta  h_t dt +B(\nabla  h_t)^2 dt + \sqrt{C}  \mc W_t,
\end{equation*}
where $A,B,C$ are constants which depend on the thermodynamical quantities of the interface, $\Delta=\partial^2_x$, $\nabla=\partial_x$ and  $\mc W_t$ is a space-time white noise. This equation is ill-posed since its solutions are similar, locally, to the Brownian motion and the main problem comes from the nonlinear term $(\nabla  h_t)^2$, which makes no sense. For a very detailed exposition on the KPZ equation we refer to \cite{Qua2}. A way to solve this equation is to consider its Cole-Hopf solution, namely, $u_t(x)=e^{\frac{B}{A}h_t(x)}$. The Cole-Hopf solution $u_t(x)$ solves the linearized version of the KPZ equation, namely, the stochastic heat equation (SHE) with multiplicative noise given by
\begin{equation*}
\label{SHE}
d u_t = A\Delta  u_t dt +\frac{BC}{A}u_t\mc W_t.
\end{equation*}
Since the equation is now linear, its solutions can be constructed and characterized by means of the constants $A,B,C$. In \cite{BG} it was proved that
these Cole-Hopf solutions can be obtained as a scaling limit of the weakly asymmetric simple exclusion process (wasep). Their approach consists in taking a microscopic analogue of the Cole-Hopf solutions, that is, they exponentiate the underling microscopic dynamics and use this new process to characterize the solutions of the SHE by means of martingale problems. The advantage of this new process is that its martingale decomposition has a compensator which is "linear", contrarily to what happens when dealing with the original process.

More recently, there has been many advances in characterizing the KPZ universality class. The approach used is to describe several functionals of asymmetric conservative systems in terms of determinantal formulas that can be solved by using the machinery of random matrix theory. This approach allows to obtain very much detailed information about the solutions of the KPZ equation, nevertheless, it depends highly on the properties of the underlying model. For more details on this we refer to \cite{C} and references therein.

Another stochastic partial differential equation related to the KPZ equation is the  stochastic Burgers equation (SBE) which can be obtained from the KPZ equation, at least formally, by taking $\mc Y_t=\nabla h_t$.  In this case, $\mc Y_t$ satisfies
\begin{equation*}
\label{KPZ}
d \mc Y_t = A\Delta \mc Y_t dt +B\nabla \mc Y_t^2 dt + \sqrt{C} \nabla \mc W_t.
\end{equation*}

In \cite{GJ3} the SBE was derived from the scaling limits of general exclusion processes in equilibrium and in \cite{GJS} the SBE is also obtained for a general class of models in a non-equilibrium scenario. The approach is to work directly with given microscopic dynamics and to characterize the solutions by means of a martingale problem. The difficulty in this approach is to make sense to the non-linear term in the SBE. To overcome that, we used a second order Boltzmann-Gibbs Principle which allows to identify the non-linear term in the SBE from the underlying dynamics. This principle is the most difficult step to achieve in this approach and the purpose of this paper is to give a simple proof of it in the case of the most classical weakly asymmetric interacting particle system, namely, the wasep. We consider the process evolving on $\mathbb{Z}$ but we remark that all the results presented here are also true when considering  the process evolving on the one-dimensional torus.

\medskip
The outline of this paper is as follows. In Section \ref{s2} we give the statement of results, we  describe the model, we give the notion of energy solutions of the SBE equation and we introduce the density fluctuation field. In Section \ref{s3}, we characterize the limiting points of the sequence of density fields by means of a martingale problem. In Section \ref{s4} we study the non-linear term in the SBE, we state and prove the second order Boltzmann-Gibbs Principle and we generalize it to any local function of the dynamics.

\section{Statement of results}
\label{s2}
\subsection{The model}
\label{s2.1}
In this section we introduce the most classical weakly asymmetric interacting particle system, namely, the wasep. Its dynamics  can informally be described as follows. Fix  a parameter $a>0$. In the
one-dimensional lattice, we allow at most one particle per site and at each bond we associate an exponential clock. Clocks associate to different bonds are independent. When one of these clocks
ring, the occupation variables at the vertices of the bond are interchanged with a certain rate.  More precisely, a particle at  $x$ jumps to $x+1$ (resp. $x-1$) at rate  $p_n:=1/2 +a/2\sqrt{n}$ (resp. $q_n:=1-p_n$) if and only if the destination site is empty.
If the destination site is occupied, then nothing
happens and the clocks restart.
\begin{center}
\begin{figure}[h!]
%\begin{tikzpicture}[background rectangle/.style=
%{fill=light-gray,rounded corners=1ex},
%show background rectangle, rounded corners,ultra thick, scale=1]
\begin{tikzpicture}[scale=0.7]
%\tikzfading[name=fade out, inner color=transparent!45,outer color=transparent!100]

%\fill [gray, path fading=fade out] (-1,-2) rectangle (16,2);

\draw[step=1cm,black,very thin] (0,0) grid (16,0);

\foreach \y in {0,...,16}{
\draw[color=gray,very thin] (\y,0)--(\y,0.2);
}

\foreach \y in {1,...,2}{
\shade[ball color=SkyBlue] (\y-0.5,0.5) circle (.3);
}

\foreach \y in {4,...,4}{
\shade[ball color=SkyBlue] (\y-0.5,0.5) circle (.3);
}

\foreach \y in {6,...,6}{
\shade[ball color=SkyBlue] (\y-0.5,0.5) circle (.3);
}
\foreach \y in {9,...,10}{
\shade[ball color=SkyBlue] (\y-0.5,0.5) circle (.3);
}

\foreach \y in {11,...,11}{
\shade[ball color=SkyBlue] (\y-0.5,0.5) circle (.3);
}
\foreach \y in {13,...,13}{
\shade[ball color=SkyBlue] (\y-0.5,0.5) circle (.3);
}

\draw [->,color=black, line width=1] (6-0.5,1.3) arc (160:20:15pt);
\draw (6,1.7) node [color=black, above right] {{$\scriptstyle{1/2 +a/2\sqrt{n}}$}};

\draw [<-,color=black, line width=1] (5-0.5,1.3) arc (160:20:15pt);
\draw (5,1.7) node [color=black, above left] {$\scriptstyle{1/2 -a/2\sqrt{n}}$};
\end{tikzpicture}
\caption{The one-dimensional weakly asymmetric  simple exclusion process.}
\label{ssep}
\end{figure}
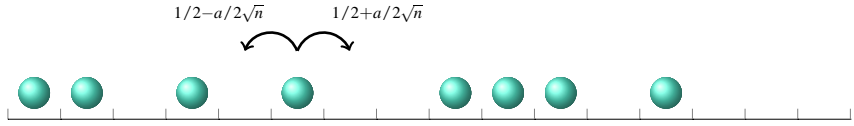
\end{center}

Formally, we denote by $\{\eta_t:=\eta_{tn^2} \,: \,t\ge 0\}$, the speeded-up, one-dimensional  weakly
asymmetric simple exclusion process with state space $\Omega = \{0,1\}^{\bb Z}$. The configurations of the state space are denoted
by the symbol $\eta$, so that $\eta(x)=1$ if the site $x$ is occupied and $\eta(x)=0$ if the site $x$ is empty. Its infinitesimal generator is denoted by $n^2\mathcal{L}_n$ where $\mathcal{L}_n$
acts on functions $f: \Omega\to\mathbb R$ as
\begin{equation*}
\mathcal{L}_n f(\eta) =\sum_{x \in \bb Z} \big\{p_n \eta(x)\big(1-\eta(x+1)\big) + q_n\eta(x+1)\big(1-\eta(x)\big) \big\} \nabla_{x,x+1} f(\eta),
\end{equation*}
where $\nabla_{x,x+1} f(\eta) = f(\eta^{x,x+1})- f(\eta)$ and for $x \in \bb Z$,
$\eta^{x,x+1}(y)=\eta(x+1)\textbf{1}_{y=x}+\eta(x)\textbf{1}_{y=x+1}+\eta(y)\textbf{1}_{y\neq x,x+1}.$

Let $\rho \in [0,1]$ and denote by $\nu_\rho$ the Bernoulli product measure on $\Omega$ with density $\rho$, which is defined as the unique measure in $\Omega$ such that the variables $\{\eta(x)\,:\, x \in \bb Z\}$ are independent and such that $\nu_\rho\{\eta(x)=1\}=\rho$ for any $x \in \bb Z$. The measures $\{\nu_\rho\,:\, \rho \in [0,1]\}$ are invariant, ergodic and reversible when $a= 0$, but for $a\neq 0$ they are no longer reversible.

We start by remarking some microscopic functions of the model that will be important in what follows. For that purpose, from now on up to the rest of this article, we fix a density $\rho \in (0,1)$ and a positive time $T$. We will consider the process $\eta_t$ with initial distribution $\nu_\rho$ and we denote by $E_\rho$ the expectation with respect to $\nu_\rho$. We denote by $\bb P_\rho$ the distribution of $\{\eta_t\,:\, t \in [0,T]\}$ in the space of c\`adl\`ag trajectories $\mc D([0,T]\,,\, \Omega)$ and we denote by $\bb E_\rho$ the expectation with respect to $\bb P_\rho$.

Fix $x\in\mathbb{Z}$. A simple computation shows that $\mathcal{L}_n(\eta(x))=j_{x-1,x}-j_{x,x+1}$, where for a bond $\{x,x+1\}$, $j_{x,x+1}$ denotes the instantaneous current  through that bond, that is, the difference between the jump rate from $x$ to $x+1$ and the jump rate from $x+1$ to $x$:
\begin{equation*}
j_{x,x+1}(\eta)=\Big(\frac{1}{2}+\frac{a}{2\sqrt{n}}\Big)\eta(x)(1-\eta(x+1))-\Big(\frac{1}{2}-\frac{a}{2\sqrt{n}}\Big)\eta(x+1)(1-\eta(x)).
\end{equation*}
We will use the following decomposition of the microscopic current
\begin{equation}
\begin{split}
j_{x,x+1}(\eta)&:=j_{x,x+1}^S(\eta)+\frac{1}{\sqrt n}j_{x,x+1}^{A}(\eta),\\
&=\frac{1}{2}(\eta(x)-\eta(x+1))+\frac{a}{2\sqrt n}(\eta(x)-\eta(x+1))^2.
\end{split}
\end{equation}
Notice that $j^S_{x,x+1}$ is written as the gradient of the function $D(\eta)=\frac{1}{2}\eta(x)$. This point will be important in what follows.

For a local function $f$, we denote by $\tilde{f}(\rho)$ its expectation with respect to $\nu_\rho$, that is,  $\tilde {f}(\rho)=\int f(\eta)\, d\nu_\rho$. Therefore,
$\tilde{j}^S(\rho)=0$ and  $\tilde{j}^A(\rho)= 2a\chi(\rho)\tilde D'(\rho)$, where
\begin{equation*}
\chi(\rho) = \rho(1-\rho) \quad \textrm{and} \quad \tilde{D}(\rho)=\frac{\rho}{2}
\end{equation*}
 are,  the static compressibility of the system and the diffusivity of the system, respectively.

\subsection{Stochastic Burgers equation}
Before introducing the Stochastic Burgers equation we need to set up some notation. Let  $\mc S(\bb R)$ be the Schwarz space of test functions and $\mc S'(\bb R)$ its  topological dual  with respect to inner product of $\mc L^2(\bb R)$. We denote by $\|F\|_2$ the $\mc L^2(\bb R)$-norm of a function $F:\mathbb{R}\to\mathbb{R}$, that is, $\|F\|^2_2=\int_\mathbb{R}F^2(x) \,dx$ . Fix $T>0$ and let $\mc C([0,T]\,,\,\mc S'(\bb R))$ be the space of continuous trajectories in $\mc S'(\bb R)$.

\begin{definition}\label{energy solution SBE}
We say that a stochastic  process $\{\mc Y_t\,:\, t \in [0,T]\}$ with trajectories in $\mc C([0,T]\,,\,\mc S'(\bb R))$ is an {\em energy solution} of the Stochastic Burgers equation
\begin{equation*}
\label{KPZ}
d \mc Y_t = A(\rho)\Delta \mc Y_t dt +B(\rho)\nabla \mc Y_t^2 dt + \sqrt{C(\rho)} \nabla \mc W_t
\end{equation*}
\begin{itemize}
\item[i)] \quad if there exists a constant $\kappa \in (0,\infty)$ such that for any $F \in \mc S(\bb R)$, any $\varepsilon >\varepsilon'>0$ and any $t \in [0,T]$ we have
\begin{equation*}
\label{energy.condition}
    \bb E\big[\big(\mc A_t^{\varepsilon}(F)-\mc A_t^{\varepsilon'}(F)\big)^2\big] \leq \kappa t \varepsilon \|\nabla F\|_2^2,
\end{equation*}
where
\begin{equation*}
 \mc A_t^\varepsilon(F):= \int_0^t \int_{\bb R} \mc Y_s(i_\varepsilon(x))^2 F'(x)\, dx \,ds
\end{equation*}
and for $x \in \bb R$, $i_\varepsilon(x,y) = \varepsilon^{-1}\mathbbm 1_{(x,x+\varepsilon]}(y)$.
 \item[ii)] \quad if the $\mc S'(\bb R)$-valued process $\{\mc A_t\,:\, t \in [0,T]\}$ defined as $\mc A_t(F): = \lim_{\varepsilon \to 0} \mc A_t^\varepsilon(F)$
 for $t \in [0,T]$ and $F \in \mc S(\bb R)$, has trajectories in $\mc C([0,T]\,,\, \mc S'(\bb R))$,
 \item[iii)]  \quad if for any function $F \in \mc S(\bb R)$, the process
\begin{equation}\label{martprobl}
\mc M_t(F)  = \mc Y_t(F) - \mc Y_0(F) - A(\rho) \int_0^t \mc Y_s(\Delta F)ds - \frac{1}{2}B(\rho) \mc A_t(F)
\end{equation}
is a continuous martingale of quadratic variation $C(\rho) t \|\nabla F\|_2^2$.
\end{itemize}
\end{definition}

\subsection{The density fluctuation field}

Recall that we have fixed a density $\rho \in (0,1)$ and we  consider the starting measure $\nu_\rho$. The density fluctuation field $\{\mc Y_t^n\, :\, t \in [0,T]\}$ is the linear functional defined on $\mc D([0,T]\,,\,\mc S'(\bb R))$  and given on $F\in\mathcal{S}(\mathbb{R})$ by
\begin{equation*}
\mc Y_t^n(F) = \frac{1}{\sqrt n} \sum_{x \in \bb Z}T_t F\Big(\frac{x}{n}\Big)(\eta_t(x)-\rho),
\end{equation*}
where $T_t F(x)=F(x - \sqrt n(\tilde{j}^A)'(\rho) t)$.
Now we explain why we remove the velocity $n^{3/2}(\tilde{j}^A)'(\rho) t=n^{3/2}(1-2\rho)at$ in the test function $F$. It is well known (see \cite{Gon} and references therein) that, at a first order, the density fluctuation fluctuation field of asymmetric systems are rigidly transported along the characteristics of the corresponding hydrodynamic equation. In order to see a non trivial temporal evolution, we should look at the density fluctuation field {\em around the characteristics} of the system. Therefore, we need to look at the density fluctuation field evolving in a time dependent reference frame as given above.

Now we are ready to state our main result. Roughly speaking, it says that, in the limit $n \to \infty$, the density fluctuation field is an energy solution of the stochastic Burgers equation.
\begin{theorem}
\label{t1}
The sequence of processes $\{\mc Y_t^n\, :\, t \in [0,T]\}_{n \in \bb N}$ is tight with respect to the Skorohod topology of $\mc D([0,T]\,,\,\mc S'(\bb R))$ and all its limit points are energy solutions of the SBE equation
\begin{equation}
\label{SBE}
d \mc Y_t = \tilde{D}'(\rho)\Delta \mc Y_t dt + \frac{1}{2}(\tilde{j}^{A})''(\rho)\nabla \mc Y_t^2 dt + \sqrt{2\chi(\rho) \tilde{D}'(\rho)} \nabla \mc W_t,
\end{equation}
where $\tilde{j}^A(\rho)= 2a\chi(\rho)\tilde D'(\rho)$.

\end{theorem}

The proof of last theorem is classical in the literature and is done in two steps. First, we prove that the sequence is tight. Second, we characterize its limit points by means of a martingale problem. Given the uniqueness of solutions of the stochastic partial differential equation, convergence follows. For tightness issues we refer the interested reader to \cite{GJ3}. Up to now, uniqueness of the KPZ/SBE has been proved in the case where the solutions are defined on a torus \cite{Hai}. Since in our case the solutions are defined on $\mathbb{R}$ our result characterizes all the limit points of the sequence as energy solutions of the SBE equation. If one proves uniqueness of energy solutions, then convergence would follow.

\section{Characterization of limits points} \label{s3}
 In this section we prove that any limit point of the sequence $\{\mc Y_t^n\,:\, t \in [0,T]\}_{n \in \bb N}$  is an energy solution of the SBE equation \eqref{SBE}. According to Definition \ref{energy solution SBE} we have to check three points. In these notes we will give a simpler proof of iii) and we refer the reader to \cite{GJ3} for the remaining points. Recall that we want to characterize the limiting field $\mathcal{Y}_t$ as a solution of \eqref{SBE} by means of a martingale problem. For that purpose we fix a test function $F\in\mathcal{S}(\mathbb{R})$ and we use the Dynkin's formula which is a tool that provides a decomposition  of functions of a Markov process as a martingale plus a compensator. More precisely, we apply the Dynkin's formula in our setting to the density fluctuation field as follows. For each $n\in\mathbb{N}$,

 \begin{equation}\label{martingaledecomposition}
 \mathcal{M}_t^n(F)=\mc Y_t^n(F) - \mc Y_0^n(F) - \int_0^t (n^2\mc L_n+\partial_s) \mc Y_s^n(F) ds
 \end{equation}
is a martingale of quadratic variation given by
\begin{equation}\label{quadvar}
\int_0^t n^2\mc{L}_n  (\mc{Y}_s^n(F))^2 - 2 \mc {Y}_s^n(F) n^2\mc{L}_n \mc{Y}_s^n(F) ds.
\end{equation}
Now we look to the compensator in \eqref{martingaledecomposition}, that is, the integral term in that equation. Recall that $\mathcal{L}_n(\eta(x))=j_{x-1,x}-j_{x,x+1}$, therefore , a summation by parts gives
\begin{equation*}
\begin{split}
n^2\mc L_n(\mc Y_s^n(F))&=\frac{n}{\sqrt n}\sum_{x\in\mathbb{Z}} \nabla_n T_s F\Big(\frac{x}{n}\Big) j_{x,x+1}(\eta_s)\\
&=\sqrt n\sum_{x\in\mathbb{Z}} \nabla_n T_s F\Big(\frac{x}{n}\Big) j^S_{x,x+1}(\eta_s)
+\sum_{x\in\mathbb{Z}} \nabla_n T_s F\Big(\frac{x}{n}\Big) j^A_{x,x+1}(\eta_s),
 \end{split}
 \end{equation*}
where $\nabla_n T_sF(x) = n (T_sF(x+1) - T_sF(x)).$

Since $j_{x,x+1}^S(\eta)$ is written as the gradient of the function $\frac{1}{2}(\eta(x)-\rho)$, the term on the left hand side of last expression can be rewritten as
\begin{equation*}
\frac{1}{2\sqrt n}\sum_{x\in\mathbb{Z}} \Delta_n T_s F\Big(\frac{x}{n}\Big)(\eta_s(x)-\rho),
 \end{equation*}
where  $\Delta_n T_sF (x) = n^2(T_sF(x+1)+ T_sF(x-1) -2 T_sF(x)).$

Since $j^A_{x,x+1}(\eta)=\frac{a}{2}(\eta(x)-\eta(x+1))^2=\frac{a}{2}(\eta(x)-2\eta(x)\eta(x+1)+\eta(x+1))$ we can write the remaining term as
\begin{equation*}
\begin{split}
a\sum_{x\in\mathbb{Z}} \nabla_n T_s F\Big(\frac{x}{n}\Big) \Big
(\frac{1}{2}\eta_s(x)-\eta_s(x)\eta_s(x+1)+\frac{1}{2}\eta_s(x+1)\Big).
 \end{split}
 \end{equation*}
On the other hand, $\partial_s \mc Y_s^n(F)$ equals to
\begin{equation*}
-a\sum_{x\in\mathbb{Z}}\partial_x T_s F\Big(\frac{x}{n}\Big)(1-2\rho)(\eta_s(x)-\rho).
 \end{equation*}
The sum of the two last terms can be written as
\begin{equation*}
\begin{split}
&a\sum_{x\in\mathbb{Z}} \nabla_n T_s F\Big(\frac{x}{n}\Big) \Big(\frac{\eta_s(x)}{2}-\eta_s(x)\eta_s(x+1)+\frac{\eta_s(x+1)}{2}-(1-2\rho)(\eta_s(x)-\rho)\Big)\\
+&a\sum_{x\in\mathbb{Z}}\Big(\nabla_n T_s F\Big(\frac{x}{n}\Big)-\partial_x T_s F\Big(\frac{x}{n}\Big)\Big)(1-2\rho)(\eta_s(x)-\rho).
 \end{split}
 \end{equation*}

 By the Taylor expansion of $T_sF$, a simple computation shows that the $L^2(\mathbb{P}_\rho)$-norm of the time integral of last term vanishes, as $n\to\infty$.
Now we look at the first term in the previous expression. Since
\begin{equation*}
\begin{split}
&\frac{\eta_s(x)}{2}-\eta_s(x)\eta_s(x+1)+\frac{\eta_s(x+1)}{2}-(1-2\rho)(\eta_s(x)-\rho)\\
=&-(\eta_s(s)-\rho)(\eta_s(x+1)-\rho)+\Big(\rho-\frac{1}{2}\Big)(\eta(x)-\eta(x+1)),
 \end{split}
 \end{equation*}
that term can be written as
\begin{equation*}
\begin{split}
-&a\sum_{x\in\mathbb{Z}} \nabla_n T_s F\Big(\frac{x}{n}\Big)(\eta_s(s)-\rho)(\eta_s(x+1)-\rho)\\
+&\frac{a}{n}\sum_{x\in\mathbb{Z}}\Delta_n T_s F\Big(\frac{x}{n}\Big)\Big(\rho-\frac{1}{2}\Big)(\eta(x)-\rho).
 \end{split}
 \end{equation*}
Notice that in all the expressions above we can introduce constants since the terms $\nabla_n T_sF$ and $\Delta_n T_sF$ add up to zero.

A simple computation shows that the $L^2(\mathbb{P}_\rho)$-norm of the time integral of the second term on the previous expression vanishes, as $n\to\infty$.
Therefore, we can write the integral part of the martingale in \eqref{martingaledecomposition} as the sum of the following terms
\begin{equation}
\label{campo I}
\mc I_t^n(F) = \int_0^t \frac{1}{2\sqrt n}\sum_{x\in\mathbb{Z}} \Delta_n T_s F\Big(\frac{x}{n}\Big)(\eta_s(x)-\rho) ds,
\end{equation}
\begin{equation}
\label{campo B}
\mc B_t^n(F) = -a\int_0^t\sum_{x\in\mathbb{Z}} \nabla_n T_s F\Big(\frac{x}{n}\Big)(\eta_s(s)-\rho)(\eta_s(x+1)-\rho) ds,
\end{equation}
plus some term $\mc R_t^n(F)$ whose $L^2(\mathbb{P}_\rho)$-norm vanishes, as $n\to\infty$.
Notice that
\begin{equation*}
\mc I_t^n(F) = \frac{1}{2} \int_0^t \mathcal{Y}^n_t(\Delta_n  F)\,ds,
\end{equation*}
and when $n\to\infty$, this term will converge to the third term on the right hand side of \eqref{martprobl}.

Now it remains to analyze the term $\mc B_t^n(F)$. This is the most difficult task and we postpone it to the next section.

Finally, the quadratic variation  of the martingale in \eqref{quadvar} can be written as
\begin{equation*}
\int_0^t\frac{1}{2n}\sum_{x\in\mathbb{Z}}\Big(\nabla_n T_s F\Big(\frac{x}{n}\Big)\Big)^2\Big((\eta_s(x)-\eta_s(x+1))^2+\frac{a}{\sqrt n}(\eta_s(x)-\eta_s(x+1))\Big) ds.
\end{equation*}
As a consequence,
\begin{equation*}
\mathbb{E}_{\rho}\Big[\Big(\mathcal{M}^n_t(F)\Big)^2\Big]\leq \int_0^t\frac{1}{2n}\sum_{x\in\mathbb{Z}}\Big(\nabla_n T_s F\Big(\frac{x}{n}\Big)\Big)^22\chi(\rho)\,ds+O\Big(\frac{1}{\sqrt n}\Big),
\end{equation*}
and $\lim_{n\to\infty}\mathbb{E}_{\rho}[(\mathcal{M}^n_t(F))^2]= t\chi(\rho)\|\nabla F\|_2^2.$ Notice that the limit of the quadratic variation is also equal to $2\chi(\rho) \tilde{D}'(\rho)\|\nabla F\|_2^2$, which matches with the strength of the noise in \eqref{SBE}.

\begin{remark}
Suppose that we are looking at a weakly asymmetric dynamics such that its symmetric part of the current $j^S_{x,x+1}$ is written as the gradient of a function $D(\eta)$ which does not allow to immediately write the corresponding term in the martingale decomposition as a function of the density field, as happens here for the term $\mathcal{I}_t^n(F)$. Then, we cannot close the term as we did above. A way to overcome this problem is to apply the classical Boltzmann-Gibbs principle, see Section \ref{s4}, which allows to do that as long as the function $D(\eta)$ is local. In that case, we would be able to write that term in the martingale decomposition as $\tilde {D}'(\rho) \int_0^t \mathcal{Y}^n_t(\Delta_n  F)\,ds$, which is exactly what we have above for the dynamics we consider here.
\end{remark}
\section{The non-linear term in SBE}\label{s4}

The main goal of this section is to analyze the term $\mc B_t^n(F)$ given in \eqref{campo B}. Looking to that field one can see that the integrand function is not written in terms of the density fluctuation field in a closed form, as happens for the term $\mathcal{I}^n_t(F)$. Therefore, we need to replace it by some function of the density field. This kind of replacement is known in the literature as the Boltzmann-Gibbs Principle. This Principle was first introduced by \cite{BR} and it says that, for any local function $f:\Omega\to\mathbb{R}$, $h:\mathbb{R}\to\mathbb{R}$ of compact support and for any $t>0$,
 \begin{equation*}
\lim_{n\rightarrow\infty}\mathbb{E}_{\rho}\Big[\int_{0}^{t}\!\!\frac{1}{\sqrt n}
\sum_{x\in{\mathbb{Z}}}h\Big(\frac{x}{n}\Big)\Big\{\tau_xf({\eta}_{s})-\tilde{f}(\rho)-\tilde{f}'(\rho)(\eta_s(x)-\rho)\Big\}\, ds\Big]^2=0.
\end{equation*}
We notice that in our case  $\tau_xf(\eta)=(\eta(x)-\rho)(\eta(x+1)-\rho)$, so that $\tilde{f}'(\rho)=0$ and the previous does give any useful information about the limit of the term $\mc B_t^n(F)$. Moreover, in $\mc B_t^n(F)$ we do not have the factor $1/\sqrt n$ in front of the sum. 
Nevertheless, the following result is the key point in order to write the term $\mc B_t^n(F)$  as a quadratic function of the field $\mathcal{Y}^n_t$. Basically it says that, when time averaged with $h$,  we can replace the function $(\eta(x)-\rho)(\eta(x+1)-\rho)$ by its conditional expectation in a box of size $\epsilon n$.

\begin{theorem}{(Second-order Boltzmann-Gibbs Principle)} \label{2bg}

For every $t>0$ and any measurable function $h:\mathbb{Z}\times[0,T]\to\mathbb{R}$,
\begin{equation*}
\lim_{\varepsilon\rightarrow 0}\lim_{n\rightarrow\infty}\mathbb{E}_{\rho}\Big[\int_{0}^{t}\!\!
\sum_{x\in{\mathbb{Z}}}h_s\Big(\frac{x}{n}\Big)\Big\{\bar{\eta}_{s}(x)\bar{\eta}_{s}(x+1)-E_\rho\Big[\bar{\eta}_s(x)\bar{\eta}_s(x+1)\Big|\eta^{\varepsilon n}_s(x)\Big]\Big\}\, ds\Big]^2=0,
\end{equation*}
where  $\bar{\eta}(x)=\eta(x)-\rho$ and $$\eta^{\varepsilon n}_s(x)=\frac{1}{\epsilon n}\sum_{y=x}^{x+\varepsilon n-1}\eta(y).$$

\end{theorem}

By the previous result, we can write
\begin{equation*}
\mc B_t^n(F) = -a\int_0^t\sum_{x\in\mathbb{Z}} \nabla_n T_s F\Big(\frac{x}{n}\Big)E_\rho\Big[\bar{\eta}_s(x)\bar{\eta}_s(x+1)\Big|\eta^{\epsilon n}_s(x)\Big] ds,
\end{equation*}
plus some error that vanishes, sending $n\rightarrow+\infty$ and then $\epsilon\rightarrow 0$.
At this point  we compute the conditional expectation appearing above. Notice that the projection of $\nu_\rho$ over the space of configurations with a fixed number of particles in the box $\{x,x+1,\cdots,\varepsilon n-1\}$ is the uniforme measure \cite{KL}, so that
\begin{equation*}
E_\rho[\eta(x)\eta(x+1)|\eta^{\varepsilon n}(x)=k]=\frac{C^{\varepsilon n-2}_{k-2}}{C^{\varepsilon n}_{k}}=\frac{k(k-1)}{\varepsilon n (\varepsilon n-1)},
\end{equation*}
and by a simple computation we get
\begin{equation}\label{condexp}
E_\rho[\bar\eta(x)\bar\eta(x+1)|\eta^{\varepsilon n}(x)]=(\bar\eta^{\varepsilon n}(x))^2-\frac{1}{\varepsilon n-1}\chi(\eta^{\varepsilon n}(x)).
\end{equation}
Since $\sum_{x\in\mathbb{Z}} \nabla_n T_s F\Big(\frac{x}{n}\Big)$ is equal to zero, we can introduce constants in the summation above and since
$(\bar\eta_s^{\varepsilon n}(x))^2=\frac{1}{n}\Big(\mathcal{Y}_s^n(\iota_\varepsilon(x/n))\Big)^2$ we write
\begin{equation}
\begin{split}
\mc B_t^n(F) =& -a\int_0^t\frac{1}{n}\sum_{x\in\mathbb{Z}} \nabla_n T_s F\Big(\frac{x}{n}\Big)\Big(\mathcal{Y}_s^n(\iota_\varepsilon(x/n))\Big)^2 \,ds\\
&+a\int_0^t\sum_{x\in\mathbb{Z}}\nabla_n T_s F\Big(\frac{x}{n}\Big)\frac{1}{\varepsilon n-1}\Big\{\chi(\eta_s^{\varepsilon n}(x))-\chi(\rho)\Big\} \, ds.
\end{split}
\end{equation}
Now, by the Cauchy-Schwarz's inequality we have
\begin{equation}
\begin{split}
&\mathbb{E}_\rho\Big[\int_0^t\sum_{x\in\mathbb{Z}}\nabla_n T_s F\Big(\frac{x}{n}\Big)\frac{1}{\varepsilon n-1}\Big\{\chi(\eta_s^{\varepsilon n}(x))-\chi(\rho)\Big\} \, ds\Big]^2\\&\leq t^2\sum_{x,y\in\mathbb{Z}}\nabla_n T_s F\Big(\frac{x}{n}\Big)\nabla_n T_s F\Big(\frac{x}{n}\Big)\frac{1}{\varepsilon^2 n^2}\Big\{\chi(\eta_s^{\varepsilon n}(x))-\chi(\rho)\Big\}\Big\{\chi(\eta_s^{\varepsilon n}(y))-\chi(\rho)\Big\}.
\end{split}
\end{equation}
Since the functions above correlate for $x$ and $y$ at a distance at most $\varepsilon n$, we can bound the previous expression by
\begin{equation}
\begin{split}
& C\frac{t^2}{\varepsilon n}\sum_{x\in\mathbb{Z}}\Big(\nabla_n T_s F\Big(\frac{x}{n}\Big)\Big)^2E_{\rho}\Big[\chi(\eta_s^{\varepsilon n}(x))-\chi(\rho)\Big]^2\\
&\leq C\frac{t^2}{\varepsilon^2 n}\|\nabla F\|_2^2,
\end{split}
\end{equation}which vanishes, as $n\to\infty$.

\subsection{Proof of the second-order Boltzmann-Gibbs Principle}

The proof of Theorem \ref{2bg} is divided into two steps. For notational convenience we consider a function $h$ that does not depend on time. We start by showing that we can replace the local function  $\bar{\eta}_{s}(x)\bar{\eta}_{s}(x+1)$ by its conditional expectation on a box of size $\ell$.

\begin{lemma} \label{oneblock}(One-block estimate)

For every $t>0$, $\ell\geq 2$ and any measurable function $h:\mathbb{Z}\to\mathbb{R}$,
\begin{equation*}
\mathbb{E}_{\rho}\Big[\int_{0}^{t}\sum_{x\in\mathbb{Z}}h(x)
V_{\ell}(\eta_{s})ds\Big]^{2}\leq Ct\frac{\ell^3}{n^2}\sum_{x\in\mathbb{Z}}h^2(x),
\end{equation*}
where $V_{\ell}(\eta_s)=\bar{\eta}_{s}(x)\bar{\eta}_{s}(x+1)-E_\rho\Big[\bar{\eta}_s(x)\bar{\eta}_s(x+1)\Big|\eta_s^\ell(x)\Big].$
\end{lemma}
\begin{proof}
Before giving the proof we introduce some notation. For two functions $f,g$ in ${L^{2}(\nu_{\rho})}$ we define the inner product $<f,-\mathcal{L}_ng>_{\rho}=-\int f(\eta)\mathcal{L}_ng(\eta)\, d\nu_{\rho}$. Let ${H}_{1}$ be the Hilbert space generated by
$L^{2}(\nu_{\rho})$ and this inner product. Denote by $\|\cdot\|_{1}$ the norm induced by this inner product and let $\|\cdot\|_{-1}$ be its
dual norm with respect to $L^{2}(\nu_{\rho})$:
\begin{equation} \label{variationalformH-1norm}
\|f\|_{-1}^2=\sup_{g\in{L^{2}(\nu_{\rho})}}\Big\{2\int f(\eta)g(\eta)\, d\nu_{\rho}-\|g\|_{1}^2\Big\}.
\end{equation}
Immediately we see that for every $f\in{H_{-1}}$, $g\in{L^{2}(\nu_{\rho})}$ and $A>0$
\begin{equation} \label{norm bound}
2\int f(\eta)g(\eta)\, d\nu_{\rho}\leq{\frac{1}{A}\|f\|^2_{-1}+A\|g\|_{1}^2}.
\end{equation}
By Proposition A1.6.1 of \cite{KL}, the   expectation in the statement of the lemma is bounded from above by
\begin{equation*}\label{KV bound}
Ct\Big\|\sum_{x\in\mathbb{Z}}h(x)V_{\ell}\Big\|_{-1}^{2},
\end{equation*}
where $C$ is a constant.

 By the variational formula for the $H_{-1}$-norm (\ref{variationalformH-1norm}) the previous expression  is equal to
\begin{equation*}
Ct \sup_{g \in L^{2}(\nu_{\rho})}\Big\{2\sum_{x\in\mathbb{Z}}h(x) \int V_{\ell}(\eta)g(\eta)\, d\nu_{\rho}-n^2<g,-\mathcal{L}_{n}g>_{\rho}\Big\}.
\end{equation*}
Now we bound $\int V_{\ell}(\eta)g(\eta)\, d\nu_{\rho}.$ We notice that by a simple computation one can prove that the adjoint of $\mc L_n$, denoted by $\mc L_n^*$ is given on
functions $f: \Omega\to\mathbb R$ by
\begin{equation*}
 \mathcal{L}_n^* f(\eta) = \sum_{x \in \bb Z} \big\{q_n \eta(x)\big(1-\eta(x+1)\big) + p_n\eta(x+1)\big(1-\eta(x)\big) \big\} \nabla_{x,x+1} f(\eta).
\end{equation*}
We denote the symmetric part of the infinitesimal generator $\mathcal L_n$ by $\mc S_n$, which is given by
\begin{equation*}
\mc S_n:=\frac{\mc L_n+\mc L_n^*}{2}.
\end{equation*}
First, we notice that for any $g\in L^2(\nu_\rho)$ it holds that $<g,-\mathcal{L}_{n}g>_{\rho}=<g,-\mathcal{S}_{n}g>_{\rho}$.
Since $(p_n+q_n) \eta(x)\big(1-\eta(x+1)\big) + (p_n+q_n)\eta(x+1)\big(1-\eta(x)\big)=\eta(x)+\eta(x+1)-2\eta(x)\eta(x+1)$, we obtain that
 \begin{equation*}
 \mathcal{S}_n f(\eta) = \sum_{x \in \bb Z} \Big\{\frac{1}{2} \eta(x)\big(1-\eta(x+1)\big) + \frac{1}{2}\eta(x+1)\big(1-\eta(x)\big) \Big\} \nabla_{x,x+1} f(\eta),
\end{equation*}
which is the infinitesimal generator of the symmetric simple exclusion process.

Before proceeding we compute $ <g,-\mathcal{S}_{n}g>_{\rho}$. By definition it equals to

\begin{equation*}
\begin{split}
 -\sum_{x\in\mathbb{Z}}\int g(\eta)\Big\{\frac{1}{2} \eta(x)\big(1-\eta(x+1)\big) + \frac{1}{2}\eta(x+1)\big(1-\eta(x)\big) \Big\} \nabla_{x,x+1} g(\eta)d\nu_\rho.
 \end{split}
\end{equation*}
Now we write it as twice its half and for each $x\in\mathbb{Z}$ and in one of the parcels we make the exchange $\eta$ into $\eta^{x,x+1}$ (for which the measure $\nu_\rho$ is invariant) to obtain
\begin{equation*}
\begin{split}
 &-\frac{1}{2}\sum_{x\in\mathbb{Z}}\int g(\eta)\Big\{\frac{1}{2} \eta(x)\big(1-\eta(x+1)\big) + \frac{1}{2}\eta(x+1)\big(1-\eta(x)\big) \Big\} \nabla_{x,x+1} g(\eta)d\nu_\rho\\
 +&\frac{1}{2}\sum_{x\in\mathbb{Z}}\int g(\eta^{x,x+1})\Big\{\frac{1}{2} \eta(x+1)\big(1-\eta(x)\big) + \frac{1}{2}\eta(x)\big(1-\eta(x+1)\big) \Big\} \nabla_{x,x+1} g(\eta)d\nu_\rho.
\end{split}
\end{equation*}
Now we organize the terms and we get
\begin{equation*}
<g,-\mathcal{S}_{n}g>_{\rho}=\frac{1 }{4}\sum_{x\in\mathbb{Z}}\int \Big( \eta(x)-\eta(x+1)\Big)^2 \Big(\nabla_{x,x+1} g(\eta)\Big)^2d\nu_\rho.
\end{equation*}
For simplicity we write $<g,-\mathcal{S}_{n}g>_{\rho}$ as $\sum_{x\in\mathbb{Z}}I_{x,x+1}(g)$, where
\begin{equation*}
I_{x,x+1}(g)=\frac{1 }{4}\int \Big( \eta(x)-\eta(x+1)\Big)^2 \Big(\nabla_{x,x+1} g(\eta)\Big)^2d\nu_\rho.
\end{equation*}

Now for fixed $x\in\mathbb{Z}$ and $\ell\geq 2$, let $\mathcal{S}_{x,\ell}$ be the restriction of $\mathcal S_n$ to the set  $\{x,x+1,\cdots,x+\ell-1\}$. Since $E_\rho[V_{\ell}|\eta^\ell(x)]=0$, then $V_{\ell}$ belongs to the image of the generator $\mathcal{S}_{x,\ell}$.
 Therefore, by (\ref{norm bound}) for each $x\in\mathbb{Z}$ and
$A_x$ a positive constant it holds that
\begin{equation*}
\int V_{\ell}(\eta)g(\eta)\, d\nu_{\rho}\leq{\frac{1}{2A_{x}}<V_{\ell},(-\mc{S}_{x,\ell})^{-1}V_{\ell}>_{\rho}+\frac{A_{x}}{2}<g,-\mc{S}_{x,\ell}g>_{\rho}}.
\end{equation*}

Now, we notice that by translation invariance it holds that
\begin{equation*}
\sum_{x\in\mathbb{Z}}<g,-\mc{S}_{x,\ell}g>_{\rho}=\sum_{x\in\mathbb{Z}}\sum_{y=x+1}^{x+\ell}I_{y,y+1}(g)
\leq \sum_{x\in\mathbb{Z}}\ell I_{x,x+1}(g) = \ell<g,-\mathcal{L}_{n}g>_{\rho}.
\end{equation*}

Therefore, taking for each $x$, $A_{x}=n^{2}(h(x))^{-1}\ell^{-1}$, the expectation in the statement of the lemma becomes bounded by
\begin{equation}\label{eqvar}
Ct\sum_{x\in\mathbb{Z}}h^2(x)\frac{\ell}{n^2}<V_{\ell},(-\mc{S}_{x,\ell})^{-1}V_{\ell}>_{\rho}.
\end{equation}
By the spectral gap inequality \cite{Qua} we have that $<V_{\ell},(-\mc{S}_{x,\ell})^{-1}V_{\ell}>_{\rho}\leq \ell^2Var_\rho[V_\ell]$, where $Var_\rho[V_\ell]$ denotes the variance of the function $V_\ell$ with respect to $\nu_\rho$. From \eqref{condexp} it is easy to see that $Var_{\rho}[V_{\ell}]\leq C$, from where \eqref{eqvar} is
bounded by
\begin{equation*}
Ct\frac{\ell^3}{n^2}\sum_{x\in\mathbb{Z}}h^2(x),
\end{equation*}
which finishes the proof.
\end{proof}
We remark that from the previous estimate we cannot take $\ell=\varepsilon n$, otherwise the error would blow up, when taking $n\to\infty$.

The second step consists replacing the conditional expectation of $\bar{\eta}_{s}(x)\bar{\eta}_{s}(x+1)$ in the box of size $\ell$ by its conditional expectation in a box of size $2\ell$.

\begin{lemma} \label{renstep1}(Renormalization step)

For every $t>0$, $\ell\geq 2$ and any measurable function $h:\mathbb{Z}\to\mathbb{R}$,
\begin{equation*}
\mathbb{E}_{\rho}\Big[\int_{0}^{t}\sum_{x\in\mathbb{Z}}h(x)
\tilde{V}_{\ell, 2\ell}(\eta_{s})ds\Big]^{2}\leq Ct\frac{\ell}{n^2}\sum_{x\in\mathbb{Z}}h^2(x),
\end{equation*}
where $\tilde V_{\ell,2\ell}(\eta_s)=E_\rho\Big[\bar{\eta}_s(x)\bar{\eta}_s(x+1)\Big|\eta_s^\ell(x)\Big]-E_\rho\Big[\bar{\eta}_s(x)\bar{\eta}_s(x+1)\Big|\eta^{2\ell}_s(x)\Big].$
\end{lemma}
\begin{proof}

Using the same arguments as in the proof of Lemma \ref{oneblock}, the expectation becomes bounded by
\begin{equation*}
Ct \sup_{g \in L^{2}(\nu_{\rho})}\Big\{2\sum_{x\in\mathbb{Z}}h(x) \int
\tilde{V}_{\ell,2\ell}(\eta)g(\eta)\, d\nu_{\rho}-n^{2}<g,-\mathcal{S}_{n}g>_{\rho}\Big\}.
\end{equation*}
Therefore, taking for each $x$, $A_{x}=n^{2}(h(x))^{-1}(2\ell)^{-1}$ together with the spectral gap inequality, last expression becomes bounded by
\begin{equation}\label{eqvar1}
Ct\sum_{x\in\mathbb{Z}}h^2(x)\frac{\ell^3}{n^2}Var_\rho[\tilde {V}_{\ell,2\ell}].
\end{equation}
Since,
\begin{equation*}
\tilde{V}_{\ell,2\ell}(\eta)=(\bar\eta^\ell(x))^2-\frac{1}{\ell-1}\chi(\eta^\ell(x))-(\bar\eta^{2\ell}(x))^2+\frac{1}{2\ell-1}\chi(\eta^{2\ell}(x)),
\end{equation*}
then a simple computation shows that $Var_\rho[\tilde V_{\ell,2\ell}]\leq{\frac{C}{\ell^2}}$, from where we get that
 \eqref{eqvar1} is bounded by
 \begin{equation}
Ct\frac{\ell}{n^2}\sum_{x\in\mathbb{Z}}h^2(x),
\end{equation}
 which finishes the proof.
\end{proof}
We notice that instead of doubling the box in the previous lemma, we can also estimate the price for replacing for the conditional expectation in a box of size $\ell $ by the condition expectation in a box of size $L$.

\begin{lemma} \label{renstep2}($L$-Renormalization step)

For every $t>0$, $\ell\geq 2$ and any measurable function $h:\mathbb{Z}\to\mathbb{R}$,
\begin{equation*}
\mathbb{E}_{\rho}\Big[\int_{0}^{t}\sum_{x\in\mathbb{Z}}h(x)
\tilde  V_{\ell, L}(\eta_{s})ds\Big]^{2}\leq Ct\frac{L^3}{n^2\ell^2}\sum_{x\in\mathbb{Z}}h^2(x),
\end{equation*}
where $\tilde V_{\ell,L}(\eta_s)=E_\rho\Big[\bar{\eta}_s(x)\bar{\eta}_s(x+1)\Big|\eta_s^\ell(x)\Big]-
E_\rho\Big[\bar{\eta}_s(x)\bar{\eta}_s(x+1)\Big|\eta^{L}_s(x)\Big].$
\end{lemma}
\begin{proof}
The only difference  with respect to the previous proof is that here we take for each $x$, $A_{x}=n^{2}(h(x))^{-1}L^{-1}$, the spectral gap inequality $<\tilde V_{\ell, L},(-\mc{S}_{x,L})^{-1}\tilde V_{\ell, L}>_{\rho}\leq L^2Var_\rho[\tilde V_{\ell,L}]$ and
in this case $Var_\rho(\tilde{V}_{\ell,L})\leq{\frac{C}{\ell^2}}$, from where the result follows.
\end{proof}

So far we have been able to replace the local function $\bar\eta(x)\bar\eta(x+1)$ by its condition expectation in a box of size $2\ell$. Now we want to increase the box in order to get to one of size $\epsilon n$. We notice that so far we cannot do that, otherwise the errors (obtained in the One-block estimate and in the Renormalization step)  explode, as $n\to\infty$. In the next lemma we compute the price to go any box.

\begin{lemma} \label{twoblocks}(Two-blocks estimate)

For every $t>0$, $\ell\geq 2$, $\ell\geq{\ell_0}$ and any measurable function $h:\mathbb{Z}\to\mathbb{R}$,
\begin{equation*}
\mathbb{E}_{\rho}\Big[\int_{0}^{t}\sum_{x\in\mathbb{Z}}h(x)
\tilde V_{\ell_0,\ell}(\eta_{s})ds\Big]^{2}\leq Ct\frac{\ell}{n^2}\sum_{x\in\mathbb{Z}}h^2(x).
\end{equation*}
\end{lemma}
\begin{proof}
We start by giving the proof in the case $\ell=2^M\ell_0$. In this case we write
\begin{equation*}
\tilde V_{\ell_0,\ell}(\eta_{s})=\sum_{i=0}^{M-1}\tilde V_{2^i\ell_0,2^{i+1}\ell_0}(\eta_s).
\end{equation*}
Therefore, by Minkowski's inequality the expectation in the statement of the lemma is bounded from above by
\begin{equation*}
\Big\{\sum_{i=1}^{M}\Big(\mathbb{E}_{\rho}\Big[\int_{0}^t\sum_{x\in\mathbb{Z}}h(x)\tilde{V}_{2^{i-1}\ell_0,2^{i}\ell_0}ds\Big]^2\Big)^{1/2}\Big\}^2.
\end{equation*}
From the Renormalization step (Lemma \ref{renstep1}) we have
\begin{equation*}
\mathbb{E}_{\rho}\Big[\int_{0}^t\sum_{x\in\mathbb{Z}}h(x)\tilde{V}_{2^{i-1}\ell_0,2^{i}\ell_0}ds\Big]^2\leq Ct\frac{2^i\ell_0}{n^2}\sum_{x\in\mathbb{Z}}h^2(x),
\end{equation*}
from where we get that last expression bounded by
\begin{equation*}
    Ct\Big\{\sum_{i=1}^{M}2^{i/2}\Big\}^2\frac{\ell_0}{n^2}\sum_{x\in\mathbb{Z}}h^2(x)\leq Ct\frac{2^M\ell_0}{n^2}\sum_{x\in\mathbb{Z}}h^2(x).
\end{equation*}
This proves the lemma for the case $\ell=2^M\ell_0$. In the other cases, we choose $M$ such that $2^M\ell_0\leq \ell \leq 2^{M+1}\ell_0$.
Then, \begin{equation*}
\begin{split}
\mathbb{E}_{\rho}\Big[\int_{0}^{t}\sum_{x\in\mathbb{Z}}h(x)
\tilde V_{\ell_0,\ell}(\eta_{s})ds\Big]^{2}\leq \Big\{\sum_{i=1}^{M}\Big(\mathbb{E}_{\rho}\Big[\int_{0}^t\sum_{x\in\mathbb{Z}}h(x)\tilde{V}_{2^{i-1}\ell_0,2^{i}\ell_0}ds\Big]^2\Big)^{1/2}\\
+\Big(\mathbb{E}_{\rho}\Big[\int_{0}^{t}\sum_{x\in\mathbb{Z}}h(x)
\tilde  V_{2^M\ell_0, \ell}(\eta_{s})ds\Big]^{2}\Big)^{1/2}\Big\}^2.
\end{split}
\end{equation*}
From the previous computations, we have that the first term on the right hand side of the previous expression is bounded from above by
$Ct\frac{\ell}{n^2}\sum_{x\in\mathbb{Z}}h^2(x).$ To bound the second term, we notice that by the $L$-Renormalization step (Lemma \ref{renstep2}) we get that
\begin{equation*}
\begin{split}
\mathbb{E}_{\rho}\Big[\int_{0}^{t}\sum_{x\in\mathbb{Z}}h(x)
\tilde  V_{2^M\ell_0, \ell}(\eta_{s})ds\Big]^{2}&\leq Ct \frac{\ell^3}{n^2(2^M\ell_0)^2}\sum_{x\in\mathbb{Z}}h^2(x)\\
&\leq Ct \frac{(2^{M+1}\ell_0)^3}{n^2(2^M\ell_0)^2}\sum_{x\in\mathbb{Z}}h^2(x)
\leq Ct\frac{\ell}{n^2}\sum_{x\in\mathbb{Z}}h^2(x).
\end{split}
\end{equation*}
which ends the proof.
\end{proof}

\subsubsection{The proof of Theorem \ref{2bg}}
From the previous computations we obtain that
the expectation appearing in the statement of the Theorem is bounded by

\begin{equation*}
Ct\frac{\varepsilon n}{n^2}\sum_{x\in\mathbb{Z}}h^2\Big(\frac{x}{n}\Big)
\end{equation*}
which converges to $Ct\varepsilon\|h\|_2^2$, as $n\to+\infty$, and then vanishes taking $\varepsilon\to 0$.

\subsection{The Boltzmann-Gibbs Principle for occupation variables}
In this section we make some remarks concerning the Boltzmann-Gibbs Principle.
To make it as general as we can let $m\in\mathbb{N}$ and $x\in\mathbb{Z}$ and let $f^x_m(\eta)=\prod_{y=x}^{x+m-1}\bar\eta(y)$. A simple computation shows that
for any $\sigma\in[0,1],$ $E_{\sigma}[f^x_m]=(\sigma-\rho)^m$.

Now we establish the Boltzmann-Gibbs Principle (Theorem \ref{2bg}) for the function $f_m^x$. We have already done it for the case $m=2$ and now we look to the other cases.

\begin{lemma} \label{moneblock}(One-block estimate)

For every $t>0$, $\ell\geq 2$, $m\in\mathbb{N}$ and any measurable function $h:\mathbb{Z}\to\mathbb{R}$,
\begin{equation*}
\mathbb{E}_{\rho}\Big[\int_{0}^{t}\sum_{x\in\mathbb{Z}}h(x)
V^m_{\ell}(\eta_{s})ds\Big]^{2}\leq Ct\frac{\ell^3}{n^2}\sum_{x\in\mathbb{Z}}h^2(x),
\end{equation*}
where $V^m_{\ell}(\eta_s)=f^x_m(\eta_s)-E_\rho\Big[f^x_m(\eta_s)\Big|\eta_s^\ell(x)\Big].$
\end{lemma}

In this case the error is the same as above, since $Var_\rho[V_\ell^m]\leq C$. Nevertheless, in the  Renormalization step, the higher the value of $m$ the less is the error.
\begin{lemma} \label{mrenstep1}(Renormalization step)

For every $t>0$, $\ell\geq 2$, $m\in\mathbb{N}$ and any measurable function $h:\mathbb{Z}\to\mathbb{R}$,
\begin{equation*}
\mathbb{E}_{\rho}\Big[\int_{0}^{t}\sum_{x\in\mathbb{Z}}h(x)
\tilde V^m_{\ell, 2\ell}(\eta_{s})ds\Big]^{2}\leq Ct\frac{\ell^{3-m}}{n^2}\sum_{x\in\mathbb{Z}}h^2(x),
\end{equation*}
where $\tilde V^m_{\ell,2\ell}(\eta_s)=E_\rho\Big[f^x_m(\eta_s)\Big|\eta_s^\ell(x)\Big]-E_\rho\Big[f_m^x(\eta_s)\Big|\eta^{2\ell}_s(x)\Big].$
\end{lemma}

\begin{proof}
The only difference  with respect to the proof in the case $m=2$ is that here  $Var_\rho[\tilde{V}^m_{\ell,2\ell}]\leq{\frac{C}{\ell^m}}$, from where the result follows.
\end{proof}

In the case of the $L$-Renormalization step we  have that

\begin{lemma} \label{mrenstep2}(L-Renormalization step)

For every $t>0$, $\ell\geq 2$, $m\in\mathbb{N}$ and any measurable function $h:\mathbb{Z}\to\mathbb{R}$,
\begin{equation*}
\mathbb{E}_{\rho}\Big[\int_{0}^{t}\sum_{x\in\mathbb{Z}}h(x)
\tilde  V^m_{\ell, L}(\eta_{s})ds\Big]^{2}\leq Ct\frac{L^3}{n^2\ell^m}\sum_{x\in\mathbb{Z}}h^2(x),
\end{equation*}
where $\tilde V^m_{\ell,L}(\eta_s)=E_\rho\Big[f^x_m(\eta_s)\Big|\eta_s^\ell(x)\Big]-
E_\rho\Big[f_m^x(\eta_s)\Big|\eta^{L}_s(x)\Big].$
\end{lemma}

\begin{lemma} \label{mtwoblocks}(Two-blocks estimate)

For every $t>0$, $\ell\geq 2$, $m\in\mathbb{N}$ and any measurable function $h:\mathbb{Z}\to\mathbb{R}$,
\begin{equation*}
\mathbb{E}_{\rho}\Big[\int_{0}^{t}\sum_{x\in\mathbb{Z}}h(x)
\tilde V^m_{\ell_0,\ell}(\eta_{s})ds\Big]^{2}\leq Ct\frac{\ell}{n^2}\sum_{x\in\mathbb{Z}}h^2(x).
\end{equation*}
\end{lemma}
\begin{proof}
In this case we notice that the only difference in the proof is that now we have the expectation bounded by
\begin{equation*}
    Ct\Big\{\sum_{i=1}^{M}2^{i(3-m)/2}\Big\}^2\frac{\ell_0^{3-m}}{n^2}\sum_{x\in\mathbb{Z}}h^2(x).
\end{equation*}
Now, if $m=1$ then last expression is bounded by
\begin{equation*}
    Ct\Big\{\sum_{i=1}^{M}2^{i}\Big\}^2\frac{\ell_0^{2}}{n^2}\sum_{x\in\mathbb{Z}}h(x)\leq Ct\frac{\ell^2}{n^2}\sum_{x\in\mathbb{Z}}h^2(x),
\end{equation*}
while if $m=3$ it is bounded by
\begin{equation*}
    Ct\Big\{\sum_{i=1}^{M}1\Big\}^2\frac{1}{n^2}\sum_{x\in\mathbb{Z}}h^2(x)=CtM^2\frac{1}{n^2}\sum_{x\in\mathbb{Z}}h^2(x).
\end{equation*}
 Since $\ell=2^M\ell_0$, then last expression equals to $Ct\frac{(\log(\ell))^2}{n^2}\sum_{x\in\mathbb{Z}}h^2(x).$ For the other cases of $m$, one notices that $\sum_{i=1}^{M}2^{i(3-m)/2}\leq \sum_{i=1}^{\infty}2^{i(3-m)/2}<\infty$ from where we get the bound
\begin{equation*}
Ct\frac{1}{n^2}\sum_{x\in\mathbb{Z}}h^2(x).
\end{equation*}
\end{proof}

We can summarize the Boltzmann-Gibbs Principle as follows:

\begin{theorem}{(General Boltzmann-Gibbs Principle)} \label{genbg}

For every $t>0$, $\ell\geq 2$, $m\in\mathbb{N}$ and any measurable function $h:[0,T]\times\mathbb{Z}\to\mathbb{R}$,
\begin{equation*}
\mathbb{E}_{\rho}\Big[\int_{0}^{t}
\frac{1}{\theta(n)}\sum_{x\in{\mathbb{Z}}}h_s\Big(\frac{x}{n}\Big)\Big\{f_m^x(\eta_s)-E_\rho\Big[f_m^x(\eta_s)\Big|\eta^{\ell}_s(x)\Big]\Big\}\, ds\Big]^2\leq{Ct\frac{c_m(\ell)}{n(\theta(n))^2}\|h\|_{2,n}^2},
\end{equation*}
where $\|h\|_{2,n}^2=\frac{1}{n}\sum_{x\in\mathbb{Z}}h^2\Big(\frac{x}{n}\Big)$ and
\begin{equation}\label{clm}
c_m(\ell)=\left\{\begin{array}{cl}
\ell^2,& \mbox{if}\,\,\, m=1\,,\\
\ell,& \mbox{if} \,\,\,m=2\,,\\
(\log(\ell))^2,& \mbox{if} \,\,\,m=3\,,\\
1,& \mbox{otherwise,}
\end{array}
\right.
\end{equation}
\end{theorem}
In order to conclude we have the following scheme.

\begin{enumerate}
\item $m=1$

\medskip

From the previous estimates we have that for $\theta(n)>\sqrt n$
\begin{equation*}
\lim_{n\to+\infty}\mathbb{E}_{\rho}\Big[\Big(\int_{0}^{t}
\frac{1}{\theta(n)}\sum_{x\in{\mathbb{Z}}}\nabla_n h\Big(\frac{x}{n}\Big)f_m^x(\eta_s)\, ds\Big)^{2}\Big]=0,
\end{equation*}
while for $\theta(n)=\sqrt n$, since $E_\rho\Big[f_m^x(\eta_s)\Big|\eta^{\varepsilon n}_s(x)\Big]=\bar\eta^{\varepsilon n}_s(x)=\frac{1}{\sqrt n}\mathcal{Y}_t^n(\iota_\varepsilon(x/n))$  we have
\begin{equation*}
\begin{split}
\int_{0}^t\frac{1}{\sqrt n}\sum_{x\in{\mathbb{Z}}}\nabla_n h\Big(\frac{x}{n}\Big)\bar\eta_s(x) \,ds\sim \int_{0}^t&\frac{1}{ n}\sum_{x\in{\mathbb{Z}}}\nabla_n h\Big(\frac{x}{n}\Big)\mathcal{Y}_t^n(\iota_\varepsilon(x/n))\, ds\\&\downarrow\\
\int_{0}^t\int_{\mathbb{R}}&\nabla h(x)\mathcal{Y}_t(\iota_\varepsilon(0))\, ds
\end{split}
\end{equation*}

\item $m=2$

\medskip

From the previous estimates, for $\theta(n)=n^{\delta}$, with $\delta>0$ we have
\begin{equation*}
\lim_{n\to+\infty}\mathbb{E}_{\rho}\Big[\Big(\int_{0}^{t}
\frac{1}{\theta(n)}\sum_{x\in{\mathbb{Z}}}\nabla_n h\Big(\frac{x}{n}\Big)f_m^x(\eta_s)\, ds\Big)^{2}\Big]=0,
\end{equation*}
while for $\theta(n)=1$, since $E_\rho\Big[f_m^x(\eta_s)\Big|\eta^{\varepsilon n}_s(x)\Big]=(\bar\eta^{\varepsilon n}_s(x))^2-\frac{1}{\varepsilon n-1}\chi(\eta^{\varepsilon n}(x))$ and since $(\bar\eta^{\varepsilon n}_s(x))^2=\frac{1}{n}(\mathcal{Y}_t^n(\iota_\varepsilon(x/n)))^2 $ we get
\begin{equation*}
\begin{split}
\int_{0}^t\sum_{x\in{\mathbb{Z}}}\nabla_n h\Big(\frac{x}{n}\Big)\bar\eta_s(x)\bar\eta_s(x+1) \,ds\sim \int_{0}^t&\frac{1}{ n}\sum_{x\in{\mathbb{Z}}}\nabla_n h\Big(\frac{x}{n}\Big)(\mathcal{Y}_t^n(\iota_\varepsilon(x/n)))^2\, ds\\&\downarrow\\
\int_{0}^t\int_{\mathbb{R}}&\nabla h(x)(\mathcal{Y}_t(\iota_\varepsilon(0)))^2\, ds.
\end{split}
\end{equation*}

\end{enumerate}

\subsection{The Boltzmann-Gibbs Principle for general functions}

In this section we rephrase the results of the previous subsection in terms of more general functions. First we introduce the notion of the degree of a function.
\begin{definition}
Let $f:\Omega\to\mathbb{R}$ be a local function and for $\sigma\in{[0,1]}$, recall that $\tilde{f}(\sigma)=E_{\sigma}[f]$.
The function $f$ is said to have degree $m\in\mathbb{N}$ if $\tilde{f}^j(\rho)=0$, for all $j=0,\cdots,m-1$ and $\tilde{f}^m(\rho)\neq {0}$.
\end{definition}

Notice that the function $f_m^x$ defined above has degree $m$.
From the previous results we have the following bounds:

\begin{corollary}
Let $m\in\mathbb{N}$ and let $f:\Omega\to\mathbb{R}$ be a local function of degree $m$ with support contained in $\{x,x+1,\cdots, x+\ell\}$, with $\ell>m$.
For every $t>0$, $\ell\geq 2$ and any measurable function $h:[0,T]\times\mathbb{Z}\to\mathbb{R}$,
\begin{equation*}
\mathbb{E}_{\rho}\Big[\int_{0}^{t}
\frac{1}{\theta(n)}\sum_{x\in{\mathbb{Z}}}h_s\Big(\frac{x}{n}\Big)\Big\{f(\eta_s)-E_\rho\Big[f(\eta_s)\Big|\eta^{\ell}_s(x)\Big]\Big\}\, ds\Big]^2\leq{Ct\frac{c_m(\ell)}{n(\theta(n))^2}\|h\|_{2,n}^2},
\end{equation*}
where
$c_m(\ell)$ was given in \eqref{clm}.
\end{corollary}

\begin{proof}
Recall the function $f_m^x$ given above. For a function $f$ with degree $m$, we define the auxiliary function $$\psi(\eta)=f(\eta)-\frac{\tilde{f}^m(\rho)}{m!}\prod_{y=x}^{x+m-1}\bar\eta(y),$$
where $\tilde{f}^m(\rho)=\frac{d^m\tilde {f}}{d\rho^m}(\rho)$ which is non zero,  since $f$ has degree $m$.
A simple computation shows that $$\tilde{\psi}(\sigma)=\tilde{f}(\sigma)-\frac{\tilde{f}^m(\rho)}{m!}(\sigma-\rho)^m,$$ from where it follows that the degree of $\psi$ is greater or equal than $m+1$.
Therefore, as a consequence of the previous results, by using the inequality $(x+y)^2\leq 2x^2+2y^2$, the fact that $c_{m+1}(\ell)\leq c_{m}(\ell)$ and also by writing

\begin{equation*}
\begin{split}
f(\eta)-E_\rho\Big[f(\eta)\Big|\eta^{\ell}(x)\Big]&=f(\eta)-\frac{\tilde{f}^m(\rho)}{m!}\prod_{y=x}^{x+m-1}\bar\eta(y)\\
&+\frac{\tilde{f}^m(\rho)}{m!}\prod_{y=x}^{x+m-1}\bar\eta(y)-E_\rho\Big[\frac{\tilde{f}^m(\rho)}{m!}\prod_{y=x}^{x+m-1}\bar\eta(y)\Big|\eta^{\ell}(x)\Big]\\
&+E_\rho\Big[\frac{\tilde{f}^m(\rho)}{m!}\prod_{y=x}^{x+m-1}\bar\eta(y)\Big|\eta^{\ell}(x)\Big]-E_\rho\Big[f(\eta)\Big|\eta^{\ell}(x)\Big]\\
&=\psi(\eta)-E_\rho\Big[\psi(\eta)\Big|\eta^{\ell}(x)\Big]\\&+\frac{\tilde{f}^m(\rho)}{m!}\prod_{y=x}^{x+m-1}\bar\eta(y)-E_\rho\Big[\frac{\tilde{f}^m(\rho)}{m!}\prod_{y=x}^{x+m-1}\bar\eta(y)\Big|\eta^{\ell}(x)\Big],
\end{split}
\end{equation*}
the proof ends.

\end{proof}

\section*{Acknowledgements}
The author thanks CNPq (Brazil) for support through the research project ``Additive functionals of particle systems" 480431/2013-2 and CMAT for support by ``FEDER" through the
``Programa Operacional Factores de Competitividade  COMPETE" and by
FCT through the project PEst-C/MAT/UI0013/2011.

\end{document}